\documentclass[12pt]{article}
\usepackage{amsmath}
\usepackage{amssymb}
\usepackage{mathrsfs}
\numberwithin{equation}{section}
\usepackage{datetime}
\usepackage{enumerate}
\usepackage{diagbox}
\usepackage[linkcolor=black,urlcolor=black,colorlinks=true]{hyperref}
\usepackage{xcolor,colortbl}
\usepackage{tabularx}
\def \beq {\begin{equation}}
\def \eeq {\end{equation}}
\def \ba {\begin{array}}
\def \ea {\end{array}}
\def \dis {\displaystyle}

\renewcommand{\ker}{\mathop{\rm Ker}}

\newcommand{\fdem}{\hspace*{\fill}~$\Box$\par\endtrivlist\unskip}

\def \div {\mbox{\rm div}\,}

\def \ep {\varepsilon}

\def \ph {\varphi}
\def \th {\theta}

\newcommand{\N}{\mathbb{N}}     
\newcommand{\Z}{\mathbb{Z}}
\newcommand{\R}{\mathbb{R}}

\newcommand{\Q}{\mathbb{Q}}

\def \cD {\mathscr{D}}
\def \cE {\mathscr{E}}

\def \cI {\mathscr{I}}

\def \cM {\mathscr{M}}

\def \sfC {\mathsf{C}}
\def \sfD {\mathsf{D}}
\def \sfT {\mathsf{T}}
\def \beq {\begin{equation}}
\def \eeq {\end{equation}}
\def \ba {\begin{array}}
\def \ea {\end{array}}
\def \bs {\bigskip}

\def \ecart {\noalign{\medskip}}

\newenvironment{proof}[1]{\textit{Proof#1.\,}}{\fdem}

\newtheorem{atheo}{Theorem}[section]

\newtheorem{alem}{Lemma}[section]
\newtheorem{arem}{Remark}[section]
\newtheorem{Aexa}{Exemple}[section]
\newenvironment{aexam}{\begin{Aexa}\rm}{\end{Aexa}}
\newtheorem{apro}[alem]{Proposition}
\newtheorem{adef}[alem]{Definition}
\title{\Large Asymptotics of ODE's flows everywhere or almost-everywhere in the torus:
from rotation sets to homogenization of transport equations
}
\author{\large Marc Briane \& Lo\"\i c Herv\'e
\\*[.1em]
\normalsize Univ Rennes, INSA Rennes,  CNRS, IRMAR - UMR 6625, F-35000 Rennes, France
\\*[.1em]
\normalsize mbriane@insa-rennes.fr \& loic.herve@insa-rennes.fr
}
\date{}
\begin{document}
\maketitle
\tableofcontents
\begin{abstract}
In this paper, we study various aspects of the ODE's flow $X$ solution to the equation $\partial_t X(t,x)=b(X(t,x))$, $X(0,x)=x$ in the $d$-dimensional torus $Y_d$, where $b$ is a regular $\Z^d$-periodic vector field from $\R^d$ in $\R^d$.
We present an original and complete picture in any dimension of all logical connections between the following seven conditions involving the field $b$: \begin{itemize}\setlength\itemsep{-.2em}
\item the everywhere asymptotics of the flow $X$,
\item the almost-everywhere asymptotics of the flow $X$,
\item the global rectification of the vector field $b$ in $Y_d$,
\item the ergodicity of the flow related to an invariant probability measure which is absolutely continuous with respect to Lebesgue's measure,
\item the unit set condition for Herman's rotation set $\sfC_b$ composed of the means of $b$ related to the invariant probability measures,
\item the unit set condition for the subset $\sfD_b$ of $\sfC_b$ composed of the means of $b$ related to the invariant probability measures which are absolutely continuous with respect to Lebesgue's measure,
\item the homogenization of the linear transport equation with oscillating data and the oscillating velocity $b(x/\ep)$ when $b$ is divergence free.
\end{itemize}
The main and surprising result of the paper is that the almost-everywhere asymptotics of the flow $X$ and the unit set condition for $\sfD_b$ are equivalent when $\sfD_b$ is assumed to be non empty, and that the two conditions turn to be equivalent to the homogenization of the transport equation when $b$ is divergence free. In contrast,  using an elementary approach based on classical tools of PDE's analysis, we extend the two-dimensional results of Oxtoby and Marchetto to any $d$-dimensional Stepanoff flow: this shows that the ergodicity of the flow may hold without satisfying the everywhere asymptotics of the flow.
\end{abstract}
\par\bs\noindent
\par\bs\noindent
\section{Introduction}
In this paper we study various aspects of the ODE's flow $X$ in the torus $Y_d$
\beq\label{bXi}
\left\{\ba{ll}
\dis {\partial X\over\partial t}(t,x)=b(X(t,x)), & t\in\R
\\ \ecart
X(0,x)=x\in \R^d,
\ea\right.
\eeq
where $b$ is a $\Z^d$-periodic vector field in $C^1(\R^d)^d$ (denoted by $b\in C^1_\sharp(Y_d)^d$), which completely determines the flow $X$.
\par
First, we are interested in the asymptotics of the flow $X$ depending on whether it can hold almost-everywhere (a.e.), or everywhere~(e.) in $Y_d$, namely
\[
\exists\,\lim_{t\to\infty}{X(t,x)\over t}\;\;\mu\mbox{-a.e. }x\in Y_d\quad\mbox{or}\quad\exists\,\lim_{t\to\infty}{X(t,x)\over t}\;\;\forall\,x\in Y_d,
\]
for some probability measure $\mu$ on $Y_d$.
If the flow $X$ is ergodic with respect to some invariant probability measure $\mu$, {\em i.e.} that $\mu$ agrees with its image measure $\mu_{X(t,\cdot)}$ for any $t\in\R$ (see Section~\ref{ss.defrec} below), then Birkhoff's theorem (see, {\em e.g.}, \cite[Theorem~1, Section~2, Chapter~1]{CFS}) ensures that
\[
\lim_{t\to\infty}{X(t,x)\over t}=\lim_{t\to\infty}\left({1\over t}\int_0^t b(X(s,x))\,ds\right)=\int_{Y_d}b(y)\,d\mu(y)\quad\mu\mbox{-a.e. }x\in Y_d.
\]
The non empty set $\cI_b$ composed of the invariant probability measures for the flow~$X$ plays a fondamental role in ergodic theory.
Associated with the set $\cI_b$, the rotations sets of \cite{MiZi} are strongly connected to the asymptotic behavior of the flow.
In particular, the compact convex Herman rotation set \cite{Her} defined by
\beq\label{Cbi}
\sfC_b:=\left\{\int_{Y_d}b(y)\,d\mu(y):\mu\in\cI_b\right\}
\eeq
characterizes the everywhere asymptotics of the flow, since by \cite[Proposition~2.1]{BrHe} we have for any $\zeta\in\R^d$,
\beq\label{Cbsini}
\sfC_b=\{\zeta\}\quad\Leftrightarrow\quad \forall\,x\in Y_d,\;\;\lim_{t\to\infty}{X(t,x)\over t}=\zeta.
\eeq
\par\noindent
We also consider the subset $\sfD_b$ of $\sfC_b$ defined by
\beq\label{Dbi}
\sfD_b:=\left\{\int_{Y_2}b(y)\,\sigma(y)\,dy:\sigma\in L^1_\sharp(Y_d)\mbox{ and }\sigma(y)\,dy\in\cI_b\right\},
\eeq
which is {\em a priori} less interesting than Herman's rotation set, since it may be empty and it is not compact in general. But surprisingly, the set $\sfD_b$ characterizes the almost-everywhere asymptotics of the flow, which is the first result of our paper. More precisely, assuming the existence of an a.e. positive invariant density function with respect to Lebesgue's measure, we prove that for any $\zeta\in\R^d$ (see Theorem~\ref{thm.Db}),
\beq\label{Dbsini}
\sfD_b=\{\zeta\}\quad\Leftrightarrow\quad \lim_{t\to\infty}{X(t,x)\over t}=\zeta\;\;\mbox{a.e. }x\in Y_d.
\eeq
\par
On the other hand, as a natural association with flow~\eqref{bXi}, we consider the linear transport equation with oscillating data
\beq\label{tranequi}
\left\{\ba{ll}
\dis {\partial u_\ep\over\partial t}(t,x)-b(x/\ep)\cdot\nabla_x u_\ep(t,x)=f(t,x,x/\ep) & \mbox{in }(0,T)\times\R^d
\\ \ecart
u_\ep(0,x)=u_0(x,x/\ep) & \mbox{for }x\in\R^d,
\ea\right.
\eeq
where $f(t,x,y)$ and $u_0(x,y)$ are suitably regular and $\Z^d$-periodic functions with respect to variable $y$.
In their famous paper~\cite{DiLi} DiPerna and Lions showed the strong proximity between ODE's flows~\eqref{bXi} and transport equations, in particular when the velocity has a good divergence. In the context of homogenization, the linear transport equation with oscillating data~\eqref{tranequi} as $\ep\to 0$ was widely studied in the literature. Tartar~\cite{Tar} proved that in general homogenization of first-order equations leads to nonlocal effects. These effects were studied carefully in~\cite{AHZ1} for equation \eqref{tranequi}. In order to avoid any anomalous effective effect, namely to get a homogenized transport equation of same nature, it is thus necessary to assume some additional condition. Assuming that the vector field $b$ is divergence free and the associated flow $X$ is ergodic, Brenier~\cite{Bre} first obtained the weak convergence of the solution $u_\ep$ to the transport equation. Following this seminal work, the homogenization of the transport equation was obtained for instance in \cite{Gol1,Gol2,HoXi,Tas} with various conditions, but which are all based on the ergodicity of the flow. Extending the result of \cite{Bre} with ergodicity arguments, Peirone~\cite{Pei} proved the convergence of the solution to the two-dimensional transport equation \eqref{tranequi} with $f(t,x,y)=0$ and $u_0(x,y)$ independent of $y$, under the sole assumption that $b$ is a non vanishing field in~$C^1_\sharp(Y_2)^2$. More recently, the homogenization of the transport equation with $f(t,x,y)=0$ and $u_0(x,y)$ independent of $y$, was derived in \cite{Bri2} (see~\cite{Bri3} for a non periodic framework) under the global rectification of the vector field $b$, which is not an ergodic condition, {\em i.e.} the existence of a $C^2$-diffeomorphism $\Psi$ on $Y_d$ and of a vector $\xi\in\R^d$ such that
\beq\label{reci}
\forall\,y\in Y_d,\quad\nabla\Psi(y)\,b(y)=\xi. 
\eeq
This result was extended in~\cite{BrHe} replacing the classical ergodic condition by the unit rotation set condition $\#\sfC_b=1$, or, equivalently, the everywhere asymptotics~\eqref{Cbsini} of the flow.
\par
In the present paper, we prove (see Theorem~\ref{thm.SNChom}) that the homogenization of transport equation \eqref{tranequi} with a divergence free velocity field, holds if, and only if, one of the equivalent conditions of \eqref{Dbsini} is satisfied. It is a quite new result beyond all the former results based on the sufficient conditions induced either by the ergodic condition, or by the unit Herman's rotation set condition. The proof of this result which is partly based on two-scale convergence \cite{Ngu,All}, clearly shows (see Remark~\ref{rem.HY}) the difference between the ergodic approach of \cite{HoXi}, and the present approach through the unit set condition \eqref{Dbsini} which turns out to be optimal.
\par
Therefore, we establish strong connections between the three following {\em a priori} foreign notions: the oscillations in the transport equation~\eqref{tranequi}, the means of $b$ only related to the invariant measures for the flow $X$ which are absolutely continuous with respect to Lebesgue's measure, and finally the almost-everywhere asymptotics of $X$.
More generally, owing to this new material we do build  the complete array of all logical connections between the following seven conditions (see Theorem~\ref{thm.7con} and Figure~\ref{array} below):
\begin{itemize}\setlength\itemsep{-.2em}
\item the global rectification~\eqref{reci} of the vector field $b$,
\item the ergodicity of the flow $X$ \eqref{bXi} related to an invariant probability measure which is absolutely continuous with respect to Lebesgue's measure,
\item the everywhere asymptotics of the flow $X$ in \eqref{Cbsini},
\item the almost-everywhere asymptotics of the flow $X$ in \eqref{Dbsini},
\item the unit set condition for Herman's rotation set $\sfC_b$ \eqref{Cbi},
\item the unit set condition for $\sfD_b$ \eqref{Dbi},
\item the homogenization of the transport equation~\eqref{tranequi} when $b$ is divergence free in $\R^d$.
\end{itemize}
In addition, the following pairs of conditions cannot be compared in general:
\begin{itemize}\setlength\itemsep{-.2em}
\item[-] the global rectification of $b$ and the ergodicity of $X$,
\item[-] the ergodicity of $X$ and the everywhere asymptotics of $X$,
\item[-] the ergodicity of $X$ and the unit set condition for $\sfC_b$.
\end{itemize}
The proof of the three last items involves the Stepanoff flow~\cite{Ste} (see Example~\ref{exa.OMS}) in which the vector field $b$ has a non empty finite zero set, and is parallel to a fixed direction $\xi\in\R^d$ with incommensurable coordinates.
Using a purely ergodic approach, Oxtoby~\cite{Oxt} and later Marchetto~\cite{Mar} proved that any two-dimensional flow homeomorphic to a Stepanoff flow admits a unique invariant probability measure $\mu$ for the flow which does not load the zero set of $b$, that $\mu$ is absolutely continuous with respect to Lebesgue's measure on $Y_2$, and finally that the flow is ergodic with respect to~$\mu$. Moreover, the set $\sfD_b$ is a unit set, but the rotation set $\sfC_b$ is a closed line set of $\R^2$, possibly not reduced to a unit set.
We extend (see Proposition~\ref{pro.OMS}) the two-dimensional results of \cite{Oxt,Mar} on the Stepanoff flow to any dimension $d\geq 2$, thanks to a new and elementary approach based on classical tools of PDE's analysis.
Finally, owing to another two-dimensional flow (see Example~\ref{exa.rhooDu} and Proposition~\ref{pro.CbneDb}) we obtain that the set $\sfD_b$ may be either empty or a singleton, while the rotation set $\sfC_b$ is a closed line set of~$\R^2$ possibly not reduced to a singleton.
\subsection{Notation}\label{ss.not}
\begin{itemize}
\item $(e_1,\dots,e_d)$ denotes the canonical basis of $\R^d$.
\item $``\cdot"$ denotes the scalar product and $|\cdot|$ the euclidian norm in $\R^d$.
\item $Y_d$, $d\geq 1$, denotes the $d$-dimensional torus $\R^d/\Z^d$, which is identified to the unit cube $[0,1)^d$ in $\R^d$.
\item $C^k_c(\R^d)$, $k\in\N\cup\{\infty\}$, denotes the space of the real-valued functions in $C^k(\R^d)$ with compact support in $\R^d$.
\item $C^k_\sharp(Y_d)$, $k\in\N\cup\{\infty\}$, denotes the space of the real-valued functions $f\in C^k(\R^d)$ which are $\Z^d$-periodic, {\em i.e.}
\beq\label{fper}
\forall\,k\in\Z^d,\ \forall\,x\in \R^d,\quad f(x+k)=f(x).
\eeq
\item The abbreviations ``a.e.'' for almost everywhere, and ``e.'' for everywhere will be used throughout the paper.
The simple mention ``a.e.'' refers to the Lebesgue measure on $\R^d$.
\item $dx$ or $dy$ denotes the Lebesgue measure on $\R^d$.
\item For a Borel measure $\mu$ on $Y_d$, extended by $\Z^d$-periodicity to a Borel measure $\tilde{\mu}$ on $\R^d$ (see definition \eqref{tmu} below), a $\tilde{\mu}$-measurable function $f:\R^d\to\R$ is said to be $\Z^d$-periodic $\tilde{\mu}$-a.e. in $\R^d$, if
\beq\label{fpermu}
\forall\,k\in\Z^d,\quad f(\cdot+k)=f(\cdot)\;\;\mbox{$\tilde{\mu}$-a.e. on }\R^d.
\eeq
\item For a Borel measure $\mu$ on $Y_d$, $L^p_\sharp(Y_d,\mu)$, $p\geq 1$, denotes the space of the $\mu$-measurable functions $f:Y_d\to\R$ such that $\int_{Y_d}|f(x)|^p\,d\mu(x)<\infty$.
\item $L^p_\sharp(Y_d)$, $p\geq 1$, simply denotes the space of the Lebesgue measurable functions $f$ in $L^p_{\rm loc}(\R^d)$, which are $\Z^d$-periodic $dx$-a.e. in $\R^d$.
\item  $\cM_{\rm loc}(\R^d)$ denotes the space of the non negative Borel measures on $\R^d$, which are finite on any compact set of $\R^d$.
\item $\cM_\sharp(Y_d)$ denotes the space of the non negative Radon measures on $Y_d$, and $\cM_p(Y_d)$ denotes the space of the probability measures on $Y_d$.
\item $\cD'(\R^d)$ denotes the space of the distributions on $\R^d$.
\item For a Borel measure $\mu$ on $Y_d$ and for $f\in L^1_\sharp(Y_d,\mu)$, $\overline{f}^{\mu}$ denotes the $\mu$-mean of $f$ on $Y_d$
\beq\label{barf}
\overline{f}^{\mu}:=\int_{Y_d}f(y)\,d\mu(y),
\eeq
which is simply denoted by $\overline{f}$ when $\mu$ is Lebesgue's measure.
The same notation is used for a vector-valued function in $L^1_\sharp(Y_d,\mu)^d$.
\item The notation $\cI_b$ in \eqref{Ib} will be used throughout the paper.
\end{itemize}
\subsection{Definitions and recalls}\label{ss.defrec}
Let $b:\R^d\to\R^d$ be a vector-valued function in $C^1_\sharp(Y_d)^d$.
Consider the dynamical system
\beq\label{bX}
\left\{\ba{ll}
\dis {\partial X\over\partial t}(t,x)=b(X(t,x)), & t\in\R
\\ \ecart
X(0,x)=x\in\R^d.
\ea\right.
\eeq
The solution $X(\cdot,x)$ to \eqref{bX} which is known to be unique (see, {\em e.g.}, \cite[Section~17.4]{HSD}) induces the dynamic flow $X$ associated with the vector field $b$, defined by
\beq\label{X}
\ba{lrll}
X: & \R\times \R^d & \to & \R^d
\\ \ecart
& (t,x) & \mapsto & X(t,x),
\ea
\eeq
which satisfies the semi-group property
\beq\label{sgroup}
\forall\,s,t\in\R,\ \forall\,x\in \R^d,\quad X(s+t,x)=X(s,X(t,x)).
\eeq
The flow $X$ is actually well defined in the torus $Y_d$, since
\beq\label{XxperY}
\forall\,t\in\R,\ \forall\,x\in \R^d,\ \forall\,k\in\Z^d,\ \quad X(t,x+k)=X(t,x)+k.
\eeq
Property~\eqref{XxperY} follows immediately from the uniqueness of the solution $X(\cdot,x)$ to \eqref{bX} combined with the $\Z^d$-periodicity of $b$.
\par
\par
A possibly signed Borel measure $\mu$ on $Y_d$ is said to be {\em invariant for the flow} $X$ if
\beq\label{invmu}
\forall\,t\in\R,\ \forall\,\psi\in C^0_\sharp(Y_d),\quad\int_{Y_d}\psi\big(X(t,y)\big)\,d\mu(y)=\int_{Y_d}\psi(y)\,d\mu(y).
\eeq
For a non negative Borel measure $\mu$ on $Y_d$, a function $f\in L^1_\sharp(Y_d,\mu)$ is said to be {\em invariant for the flow} $X$ with respect to $\mu$, if
\beq\label{invf}
\forall\,t\in\R,\quad f\circ X(t,\cdot)=f(\cdot)\;\;\mbox{ $\mu$-a.e. in $Y_d$}.
\eeq
The flow $X$ is said to be {\em ergodic} with respect to some invariant probability measure $\mu$, if
\beq\label{Xerg}
\forall\,f\in L^1_\sharp(Y_d,\mu),\mbox{ invariant for $X$ w.r.t. $\mu$},\quad f=\overline{f}^\mu\mbox{ $\mu$-a.e. in $Y_d$}.
\eeq
Then, define the set
\beq\label{Ib}
\cI_b:=\big\{\mu\in\cM_p(Y_d): \mu\mbox{ invariant for the flow }X\big\},
\eeq
where $\cM_p(Y_d)$ is the set of probability measures on $Y_d$.
From the set of invariant probability measures we define the so-called Herman \cite{Her} rotation set
\beq\label{Cb}
\sfC_b:=\left\{\overline{b}^\mu=\int_{Y_2}b(y)\,d\mu(y):\mu\in\cI_b\right\},
\eeq
and its subset
\beq\label{Db}
\sfD_b:=\left\{\overline{\sigma\,b}=\int_{Y_2}b(y)\,\sigma(y)\,dy:\sigma\in L^1_\sharp(Y_d)\mbox{ and }\sigma(y)\,dy\in\cI_b\right\}
\eeq
which is restricted to the invariant probability measures which are absolutely continuous with respect to Lebesgue's measure. If there is no such invariant measure, then the set $\sfD_b$ is empty (see Remark~\ref{rem.Dbe}).
\par
We have the following characterization of an invariant measure known as Liouville's theorem, which can also be regarded as a divergence-curl result with measures (see \cite[Proposition~2.2]{BrHe} and \cite[Remark~2.2]{BrHe} for further details).
\begin{apro}[Liouville's theorem]\label{pro.divcurl}
Let $b\in C^1_\sharp(Y_d)^d$, and let $\mu\in\cM_\sharp(Y_d)$. We define the Borel measure $\tilde{\mu}\in\cM_{\rm loc}(\R^d)$ on~$\R^d$ by
\beq\label{tmu}
\int_{\R^d}\ph(x)\,d\tilde{\mu}(x)=\int_{Y_d} \ph_\sharp(y)\,d\mu(y),\quad\mbox{where}\quad \ph_\sharp(\cdot):=\sum_{k\in\Z^d}\ph(\cdot+k)
\quad\mbox{for }\ph\in C^0_c(\R^d).
\eeq
Then, the three following assertions are equivalent:
\begin{enumerate}[$(i)$]
\item $\mu$ is invariant for the flow $X$, {\em i.e.} \eqref{invmu} holds,
\item $\tilde{\mu}\,b$ is divergence free in the space $\R^d$, {\em i.e.}
\beq\label{dtmub=0}
\div(\tilde{\mu}\,b)=0\quad\mbox{in }\cD'(\R^d),
\eeq
\item $\mu\,b$ is divergence free in the torus $Y_d$, {\em i.e.}
\beq\label{dmub=0}
\forall\,\psi\in C^0_\sharp(Y_d),\quad \int_{Y_d} b(y)\cdot\nabla\psi(y)\,d\mu(y)=0.
\eeq
\end{enumerate}
\end{apro}
\begin{arem}\label{rem.tmu}
Since any function $\psi\in C^\infty_\sharp(Y_d)$ can be represented as a function $\varphi_\sharp$ for a suitable function $\ph\in C^\infty_c(\R^d)$ (see \cite[Lemma~3.5]{Bri1}), we deduce that the mapping
\[
\ba{rll}
\cM_\sharp(Y_d) & \rightarrow & \dis \left\{\nu\in\cM_{\rm loc}(\R^d):\forall\,\varphi\in C^0_c(\R^d),\ \varphi_\sharp=0\Rightarrow\int_{\R^d}\varphi(x)\,d\nu(x)=0\right\}
\\ \ecart
\mu & \mapsto & \tilde{\mu}
\ea
\]
is bijective. Therefore, the measure $\tilde{\mu}$ of \eqref{tmu} completely characterizes the measure~$\mu$.
\end{arem}
\par
By virtue of \cite[Proposition~2.1]{BrHe} (see also \cite{MiZi}) Herman's set $\sfC_b$ satisfies the following result.
\begin{apro}[\cite{BrHe,MiZi}]\label{pro.Cb}
Let $b\in C^1_\sharp(Y_d)^d$. Then, we have for any $\zeta\in\R^d$,
\beq\label{Cbsin}
\sfC_b=\{\zeta\}\quad\Leftrightarrow\quad \forall\,x\in Y_d,\;\;\lim_{t\to\infty}{X(t,x)\over t}=\zeta.
\eeq
\end{apro}
\section{The rotation subset $\sfD_b$}
\noindent
We have the following characterization of the singleton condition satisfied by $\sfD_b$, which has to be compared to the one satisfied by $\sfC_b$ in Proposition~\ref{pro.Cb} above.
\begin{atheo}\label{thm.Db}
Let $b\in C^1_\sharp(Y_b)^d$ be such that there exists an a.e. positive function $\sigma_0\in L^1_\sharp(Y_d)$ with $\overline{\sigma_0}=1$, satisfying ${\rm div}(\sigma_0\,b)=0$ in $\R^d$. Then, the flow $X$ associated with $b$ satisfies for any $\zeta\in\R^d$, 
\beq\label{Dbsin}
\sfD_b=\{\zeta\}\quad\Leftrightarrow\quad \lim_{t\to\infty}{X(t,x)\over t}=\zeta,\;\;\mbox{a.e. }x\in Y_d.
\eeq
\end{atheo}
\begin{proof}{}
First of all, by virtue of the Birkhoff theorem applied with the invariant measure $\sigma_0(x)\,dx$ with the a.e. positive function $\sigma_0\in L^1_\sharp(Y_d)$, combined with the uniform boundedness of $X(t,x)/t$ for $t\in\R$ and $x\in Y_d$, there exists a function $\xi\in L^\infty_\sharp(Y_d)$ which is invariant for the flow $X$ with respect to Lebesgue's measure, such that
\[
\lim_{t\to\infty}{X(t,x)\over t} = \xi(x)\quad\mbox{a.e. }x\in Y_d.
\]
Hence, by Lebesgue's theorem we get that for any invariant measure $\sigma(x)\,dx$ with $\sigma\in L^1_\sharp(Y_d)$,
\beq\label{sixi}
\ba{ll}
\dis \int_{Y_d}b(x)\,\sigma(x)\,dx
& \dis = \lim_{t\to\infty}{1\over t} \int_0^t\left(\int_{Y_d}b(X(s,x))\,\sigma(x)\,dx\right)ds
\\ \ecart
& \dis = \int_{Y_d}\lim_{t\to\infty}\left({X(t,x)-x\over t}\right)\sigma(x)\,dx
\\ \ecart
& \dis = \int_{Y_d}\xi(x)\,\sigma(x)\,dx.
\ea
\eeq
$(\Rightarrow)$ Assume that $\sfD_b=\{\zeta\}$ for some $\zeta\in\R^d$. Then, we have for any invariant measure $\sigma(x)\,dx$ with $\sigma\in L^1_\sharp(Y_d)$,
\[
\int_{Y_d}b(x)\,\sigma(x)\,dx=\zeta\int_{Y_d}\sigma(x)\,dx,
\]
which by \eqref{sixi} implies that
\beq\label{xize}
\int_{Y_d}(\xi(x)-\zeta)\,\sigma(x)\,dx = 0.
\eeq
On the other hand, since the non negative and the non positive parts $(\xi-\zeta)^\pm$ of $\xi-\zeta$ are also invariant functions for the flow $X$ with respect to Lebesgue's measure, by Lemma~\ref{lem.invdenmea} below the measures $(\xi(x)-\zeta)^\pm\,\sigma_0(x)\,dx$ are invariant for $X$.
Therefore, putting the measures $\;\sigma(x)\,dx=(\xi(x)-\zeta)^\pm\,\sigma_0(x)\,dx\;$ in equality \eqref{xize} we get that
\[
\int_{Y_d}(\xi(x)-\zeta)\,(\xi(x)-\zeta)^\pm\,\sigma_0(x)\,dx = \pm \int_{Y_d}\big[(\xi(x)-\zeta)^\pm\big]^2\sigma_0(x)\,dx =  0,
\]
which due to the $a.e.$ positivity of $\sigma_0$, implies the right hand-side of \eqref{Dbsin}.
\par\medskip\noindent
$(\Leftarrow)$ Conversely, we deduce immediately from \eqref{sixi} that for any invariant measure $\sigma(x)\,dx$ with $\sigma\in L^1_\sharp(Y_d)$,
\[
\int_{Y_d}b(x)\,\sigma(x)\,dx = \zeta\int_{Y_d}\sigma(x)\,dx,
\]
which yields $\sfD_b = \{\zeta\}$.
\end{proof}
\begin{alem}\label{lem.invdenmea}
Let $b\in C^1_\sharp(Y_d)^d$ be a vector field in $\R^d$ such that there exists an a.e. positive function $\sigma_0\in L^1_\sharp(Y_d)$ with $\overline{\sigma_0}=1$, satisfying ${\rm div}(\sigma_0\,b)=0$ in $\R^d$. Then, a function $f$ in $L^\infty_\sharp(Y_d)$ is invariant for the flow $X$ with respect to Lebesgue's measure if, and only if, the signed measure $f(x)\,\sigma_0(x)\,dx$ is invariant for $X$.
\end{alem}
\begin{proof}{}
First of all, for any $t\in\R$, $X(t,\cdot)$ is a $C^1$-diffeomorphism on $\R^d$ with reciprocal $X(-\,t,\cdot)$, as a consequence of the semi-group property \eqref{sgroup} satisfied by the flow $X$. Moreover, by virtue of Liouville's theorem the jacobian determinant of $X(t,\cdot)$ is given by
\beq\label{jacX}
\forall\,t\in\R,\ \forall\,x\in Y_d,\quad J(t,x):=\det\big(\nabla_x X(t,x)\big)=\exp\left(\int_0^t({\rm div}\,b)(X(s,x))\,ds\right).
\eeq
Since by Proposition~\ref{pro.divcurl} the measure $\widetilde{\sigma_0(x)\,dx}=\sigma_0(x)\,dx$ (due to the $\Z^d$-periodicity of~$\sigma_0$) is invariant for the flow $X$, we have for any function $\varphi\in C^0_c(\R^d)$ and any $t\in\R$,
\[
\varphi_\sharp(X(-t,\cdot))=\big(\varphi(X(-t,\cdot))\big)_\sharp\quad\mbox{by \eqref{XxperY}},
\]
and
\[
\ba{l}
\dis \int_{\R^d}\varphi(x)\,\sigma_0(x)\,dx =\int_{Y_d}\varphi_\sharp(x)\,\sigma_0(x)\,dx
\\ \ecart
\dis =\int_{Y_d}\varphi_\sharp(X(-t,x))\,\sigma_0(x)\,dx =\int_{Y_d}\big(\varphi(X(-t,x))\big)_\sharp\,\sigma_0(x)\,dx
\\ \ecart
\dis =\int_{\R^d}\varphi(X(-t,x))\,\sigma_0(x)\,dx \underbrace{=}_{x=X(t,y)}\int_{\R^d}\varphi(y)\,\sigma_0(X(t,y))\,J(t,y)\,dy.
\ea
\]
This implies that the jacobian determinant $J(t,\cdot)$ satisfies the relation
\beq\label{JacX}
\forall\,t\in\R,\quad \sigma_0(X(t,y))\,J(t,y)=\sigma_0(y)\;\;\mbox{a.e. }y\in\R^d.
\eeq
Now, let $f\in L^\infty_\sharp(Y_d)$. From~\eqref{JacX} we deduce that for any function $\varphi\in C^0_c(\R^d)$ and any $t\in\R$,
\[
\ba{ll}
\dis \int_{\R^d}\varphi(X(-\,t,x))\,f(x)\,\sigma_0(x)\,dx
& \dis \underbrace{=}_{x=X(t,y)}\int_{\R^d}\varphi(y)\,f(X(t,y))\,\sigma_0(X(t,y))\,J(t,y)\,dy
\\ \ecart
& \dis =\int_{\R^d}\varphi(y)\,f(X(t,y))\,\sigma_0(y)\,dy.
\ea
\]
By virtue of Remark~\ref{rem.tmu} combined with the $\Z^d$-periodicity of the function $f$, the former equality also reads as
\beq\label{}
\forall\,\psi\in C^0_\sharp(Y_d),\ \forall\,t\in\R,\quad
\int_{Y_d}\psi(X(-\,t,x))\,f(x)\,\sigma_0(x)\,dx=\int_{Y_d}\psi(x)\,f(X(t,x))\,\sigma_0(x)\,dx.
\eeq
Therefore, due to the a.e. positivity of $\sigma_0$, the function $f\in L^\infty_\sharp(Y_d)$ is invariant for the flow~$X$ with respect to Lebesgue's measure, {\em i.e.} $f(X(\cdot,x))=f(x)$ a.e. $x\in Y_d$, if, and only if, the signed measure $f(x)\,\sigma_0(x)\,dx$ is invariant for the flow $X$.
\end{proof}
\section{A NSC for homogenization of the transport equation}
\noindent
First of all, recall the definition of the two-scale convergence introduced by Nguetseng~\cite{Ngu} and Allaire~\cite{All}, which is easily extended to the time dependent case.
\begin{adef}\label{def.2s}
Let $T\in(0,\infty)$.
\begin{itemize}
\item[$a)$] A sequence $u_\ep(t,x)$ in $L^2((0,T)\times\R^d)$ is said to {\em two-scale converge} to a function $U(t,x,y)$ in
$L^2([0,T]\times\R^d;L^2_\sharp(Y_d))$, if we have for any function $\varphi\in C^0_c([0,T]\times\R^d;C^0_\sharp(Y_d))$ with compact support in $[0,T]\times\R^d\times Y_d$,
\beq\label{2scon}
\lim_{\ep\to 0}\int_{(0,T)\times\R^d}\kern-.2cm u_\ep(t,x)\,\varphi(t,x,x/\ep)\,dtdx
=\int_{(0,T)\times\R^d\times Y_d}\kern-.2cm U(t,x,y)\,\varphi(t,x,y)\,dtdxdy,
\eeq
\item[$b)$] According to \cite[Definition~1.4]{All} any function $\psi(t,x,y)\in C^0_c([0,T]\times\R^d;L^2_\sharp(Y_d))$ with compact support in $[0,T]\times\R^d\times Y_d$, is said to be an {\em admissible function for two-scale convergence}, if $(t,x)\mapsto \psi(t,x,x/\ep)$ is Lebesgue measurable and
\beq\label{admtf}
\lim_{\ep\to 0}\int_{(0,T)\times\R^d}\psi^2(t,x,x/\ep)\,dtdx=\int_{(0,T)\times\R^d\times Y_d}\psi^2(t,x,y)\,dtdxdy.
\eeq
\end{itemize}
\end{adef}
Then, we have the following two-scale convergence compactness result.
\begin{atheo}[\cite{All}, Theorem~1.2, Remark~1.5]
Any sequence $u_\ep(t,x)$ which is bounded in $L^2((0,T)\times\R^d)$ two-scale converges, up to extract a subsequence, to some function $U(t,x,y)$ in $L^2((0,T)\times\R^d;L^2_\sharp(Y_d))$.
Moreover, equality~\eqref{2scon} holds true for any admissible function \eqref{admtf}.
\end{atheo}
Let $b(y)\in C^1_\sharp(Y_d)^d$ be a vector field, let $u_0(x,y)\in C^0_c(\R^d;L^2_\sharp(Y_d))$ be an admissible function with compact support in $\R^d\times Y_d$, and let $f(t,x,y)\in C^0_c([0,T]\times\R^d;L^\infty_\sharp(Y_d))$ be an admissible function with compact support in $[0,T]\times\R^d\times Y_d$.
Consider the linear transport equation with oscillating data
\beq\label{tranequ}
\left\{\ba{ll}
\dis {\partial u_\ep\over\partial t}(t,x)-b(x/\ep)\cdot\nabla_x u_\ep(t,x)=f(t,x,x/\ep) & \mbox{in }(0,T)\times\R^d
\\ \ecart
u_\ep(0,x)=u_0(x,x/\ep) & \mbox{for }x\in\R^d,
\ea\right.
\eeq
which by \cite[Proposition~II.1, Theorem~II.2]{DiLi} has a unique solution in $L^\infty((0,T);L^2(\R^d))$.
\par
\noindent
We have the following criterion for the homogenization of equation \eqref{tranequ}.
\begin{atheo}\label{thm.SNChom}
Let $b$ be a divergence free vector field in $C^1_\sharp(Y_d)^d$, and let $X$ be the flow \eqref{bX} associated with~$b$.
Then, we have the equivalence of the two following assertions:
\begin{itemize}
\item[$(i)$] There exists $\zeta\in\R^d$ such that the flow $X$ satisfies the asymptotics
\beq\label{asyX}
\lim_{t\to\infty}{X(t,x)\over t}=\zeta,\;\;\mbox{a.e. }x\in Y_d,
\eeq
or, equivalently, $\sfD_b=\{\zeta\}$.
\item[$(ii)$] There exists $\zeta\in\R^d$ such that for any admissible functions $u_0(x,y)\in C^0_c(\R^d;L^2_\sharp(Y_d))$ with compact support in $[0,T]\times\R^d$, and $f(t,x,y)\in C^0_c([0,T]\times\R^d;L^\infty_\sharp(Y_d))$ with compact support in $[0,T]\times\R^d\times Y_d$, the solution $u_\ep$ to \eqref{tranequ} converges weakly in $L^\infty((0,T);L^2(\R^d))$ to the solution $u(t,x)$ to the transport equation
\beq\label{homtranequ}
\left\{\ba{ll}
\dis {\partial u\over\partial t}(t,x)-\zeta\cdot\nabla_x u(t,x)=\overline{f(t,x,\cdot)} & \mbox{in }(0,T)\times\R^d
\\ \ecart
u(0,x)=\overline{u_0(x,\cdot)} & \mbox{for }x\in\R^d.
\ea\right.
\eeq
\end{itemize}
Moreover, in both cases we have $\zeta=\overline{b}$.
\end{atheo}
\begin{proof}{ of Theorem~\ref{thm.SNChom}}
\par\smallskip\noindent
{\it $(i)\Rightarrow(ii).$}
First of all, note that, since $b$ is divergence free in $\R^d$, by Proposition~\ref{pro.divcurl} Lebesgue's measure is an invariant probability measure for the flow~$X$ associated with $b$, which implies that $\overline{b}\in\sfD_b=\{\zeta\}$ and $\zeta=\overline{b}$.
\par
Now, let $u_0(x,y)\in C^0_c(\R^d;L^2_\sharp(Y_d))$ be an admissible function with compact support in $[0,T]\times\R^d$, and let $f(t,x,y)\in C^0_c([0,T]\times\R^d;L^\infty_\sharp(Y_d))$ be an admissible function whose support is contained in $[0,T]\times K$, $K$ being a compact set of $\R^d$.
\par\noindent
Denote $b_\ep(x):=b(x/\ep)$ which is divergence free in $\R^d$, and denote $f_\ep(t,x):=f(t,x,x/\ep)$ which is uniformly bounded in $[0,T]\times\R^d$ and is compactly supported in $[0,T]\times K$. Formally, multiplying \eqref{tranequ} by $u_\ep(t,x)$, integrating by parts over $\R^d$ and using Cauchy-Schwarz inequality, we get that for any $t\in(0,T)$,
\[
\ba{l}
\dis {1\over 2}\,{d\over dt}\left(\int_{\R^d}u_\ep^2(t,x)\,dx\right)={1\over 2}\,{d\over dt}\left(\int_{\R^d}u_\ep^2(t,x)\,dx\right)-{1\over 2}\,\int_{\R^d}{\rm div}(b_\ep)(x)\,u_\ep^2(t,x)\,dx
\\ \ecart
\dis =\int_{K}f_\ep(t,x)\,u_\ep(t,x)\,dx\leq C_f\left(\int_{\R^d}u_\ep^2(t,x)\,dx\right)^{1/2},
\ea
\]
where $C_f$ is a non negative constant only depending on $f$.
This can be justified following the proof of \cite[Proposition~II.1]{DiLi}.
Hence, we deduce the estimate 
\beq\label{estue}
\|u_\ep(t,\cdot)\|_{L^2(\R^d)}\leq \|u_\ep(0,\cdot)\|_{L^2(\R^d)}+C_f\,T\quad\mbox{a.e. }t\in(0,T).
\eeq
Therefore, estimate \eqref{estue} combined with (recall that the admissible function $\psi(t,x,y)=u_0(x,y)$ satisfies~\eqref{admtf})
\[
\lim_{\ep\to 0}\|u_\ep(0,\cdot)\|_{L^2(\R^d)}=\|u_0(x,y))\|_{L^2(\R^d\times Y_d)},
\]
implies that the sequence $u_\ep$ is bounded in $L^\infty((0,T);L^2(\R^d))$.
Then, up to a subsequence, $u_\ep(t,x)$ two-scale converges to some function $U(t,x,y)\in L^2([0,T]\times\R^d;L^2_\sharp(Y_d))$, and $u_\ep(t,x)$ converges weakly in $L^2((0,T)\times\R^d)$ to the mean
\beq\label{uU}
u(t,x):=\overline{U(t,x,\cdot)}=\int_{Y_d}U(t,x,y)\,dy\quad\mbox{for a.e. }(t,x)\in(0,T)\times\R^d.
\eeq
\par
Next, we follow the two-scale procedure of the proof of \cite[Theorem~2.1]{HoXi}. Putting the test function $\varphi(t,x)\in C^1_c([0,T)\times\R^d)$ in the weak formulation of \eqref{tranequ}, and integrating by parts we have
\[
\ba{l}
\dis -\int_{(0,T)\times\R^d}{\partial\varphi\over\partial t}(t,x)\,u_\ep(t,x)\,dtdx-\int_{\R^d}\varphi(0,x)\,u_0(x,x/\ep)\,dx
\\ \ecart
\dis +\int_{(0,T)\times\R^d} b(x/\ep)\cdot\nabla_x\varphi(t,x)\,u_\ep(t,x)\,dtdx=\int_{\R^d}\varphi(t,x)\,f(t,x,x/\ep)\,dtdx.
\ea
\]
Then, passing to the two-scale limit and using that $u_0(x,y)$ and $f(t,x,y)$ are admissible functions for two-scale convergence, we get that
\[
\ba{l}
\dis -\int_{(0,T)\times\R^d\times Y_d}{\partial\varphi\over\partial t}(t,x)\,U(t,x,y)\,dtdxdy-\int_{\R^d\times Y_d}\varphi(0,x)\,u_0(x,y)\,dxdy
\\ \ecart
\dis +\int_{(0,T)\times\R^d\times Y_d} b(y)\cdot\nabla_x\varphi(t,x)\,U(t,x,y)\,dtdxdy=\int_{(0,T)\times\R^d\times Y_d}\varphi(t,x)\,f(t,x,y)\,dtdxdy,
\ea
\]
or, equivalently, by Fubini's theorem
\beq\label{2shomequ}
\ba{l}
\dis -\int_{(0,T)\times\R^d}{\partial\varphi\over\partial t}(t,x)\,u(t,x)\,dx-\int_{\R^d}\varphi(0,x)\,\overline{u_0(x,\cdot)}\,dx
\\ \ecart
\dis +\int_{(0,T)\times\R^d}\overline{U(t,x,\cdot)\,b}\cdot\nabla_x\varphi(t,x)\,\,dtdx=\int_{\R^d\times Y_d}\varphi(t,x)\,\overline{f(t,x,\cdot)}\,dtdx.
\ea
\eeq
Similarly, passing to the two-scale limit with the admissible test function $\ep\,\varphi(t,x)\,\psi(x/\ep)$ for any $\varphi(t,x)\in C^1_c([0,T)\times\R^d)$ and any $\psi\in C^1_\sharp(Y_d)$, we get that
\[
\ba{l}
\dis \int_{(0,T)\times\R^d\times Y_d} \varphi(t,x)\,b(y)\cdot\nabla_y\psi(y)\,U(t,x,y)\,dtdxdy
\\ \ecart
\dis =\int_{(0,T)\times\R^d}\varphi(t,x)\left(\int_{Y_d}U(t,x,y)\,b(y)\cdot\nabla_y\psi(y)\,dy\right)dtdx=0,
\ea
\]
which by Proposition~\ref{pro.divcurl} implies that
\beq\label{Ubdiv0}
{\rm div}_y(U(t,x,\cdot)\,b)=0\;\;\mbox{in }\cD'(\R^d),\quad\mbox{a.e. }(t,x)\in(0,T)\times\R^d.
\eeq
Then, applying Lemma~\ref{lem.invdenmea} with $\sigma_0=1$, for a.e. $(t,x)\in(0,T)\times\R^d$, the function $U(t,x,\cdot)$ is an invariant function for the flow $X$ associated with $b$ related to Lebesgue's measure, and so are the positive and negative parts $U^\pm(t,x,\cdot)$ of $U(t,x,\cdot)$. Hence, again by Lemma~\ref{lem.invdenmea} the measures $U^\pm(t,x,y)\,dy$ are invariant for $X$, which by the definition~\eqref{Db} of $\sfD_b=\{\zeta\}$, implies that
\beq\label{Upm}
\overline{U^\pm(t,x,\cdot)\,b}=\int_{Y_d}b(y)\,U^\pm(t,x,y)\,dy=\left(\int_{Y_d}U^\pm(t,x,y)\,dy\right)\zeta\quad\mbox{a.e.}\,(t,x)\in(0,T)\times\R^d.
\eeq
From \eqref{Upm} and \eqref{uU} we deduce that
\beq\label{Uuzeta}
\overline{U(t,x,\cdot)\,b}=\overline{U(t,x,\cdot)}\;\zeta=u(t,x)\,\zeta\quad\mbox{a.e. }(t,x)\in(0,T)\times\R^d.
\eeq
Putting this equality in the weak formulation \eqref{2shomequ} we get that for any $\varphi(t,x)\in C^1_c([0,T)\times\R^d)$,
\[
\ba{l}
\dis -\int_{(0,T)\times\R^d}{\partial\varphi\over\partial t}(t,x)\,u(t,x)\,dx-\int_{\R^d}\varphi(0,x)\,\overline{u_0(x,\cdot)}\,dx
\\ \ecart
\dis +\int_{(0,T)\times\R^d}u(t,x)\,\zeta\cdot\nabla_x\varphi(t,x)\,\,dtdx=\int_{\R^d\times Y_d}\varphi(t,x)\,\overline{f(t,x,\cdot)}\,dtdx,
\ea
\]
which is the weak formulation of the homogenized transport equation \eqref{homtranequ}.
\par\medskip\noindent
{\it $(ii)\Rightarrow(i)$.} First of all, note that the set $\sfD_b$ contains the mean $\overline{b}$, since by the free divergence of $b$ and by Proposition~\ref{pro.divcurl}, Lebesgue's measure is an invariant probability measure for the flow~$X$ associated with $b$.
\par
Now, let us prove that any invariant probability measure $\sigma(x)\,dx$ with $\sigma\in L^1_\sharp(Y_d)$, for the flow $X$ satisfies the equality $\overline{\sigma\,b}=\zeta$, which will yield the desired equality $\sfD_b=\{\zeta\}$.
To this end, let us first show this for any invariant probability measure $v(x)/\overline{v}\,dx$ with $v\in L^\infty_\sharp(Y_d)$.
By virtue of Proposition~\ref{pro.divcurl} such a function $v$ is solution to the equation
\beq\label{divvb=0}
{\rm div}(v\,b)=b\cdot\nabla v=0\quad\mbox{in }\cD'(\R^d).
\eeq
Let $\theta\in C^1_c(\R^d)$, and define for $\ep>0$ the function $u_\ep\in C^1([0,T];C^1_c(\R^d))$ by
\[
u_\ep(t,x):=\theta(x+t\,\zeta)\,v(x/\ep)\quad\mbox{for }(t,x)\in[0,T]\times\R^d,
\]
where $\zeta$ is the vector involving in the homogenized equation \eqref{homtranequ}.
By \eqref{divvb=0} we have
\[
\ba{l}
\dis {\partial u_\ep\over\partial t}(t,x)-b(x/\ep)\cdot\nabla_x u_\ep(t,x)
\\ \ecart
\dis =v(x/\ep)\,\zeta\cdot\nabla_x\theta(x+t\,\zeta)-v(x/\ep)\,b(x/\ep)\cdot\nabla_x\theta(x+t\,\zeta)-1/\ep\,\theta(x+t\,\zeta)\,(b\cdot\nabla_y v)(x/\ep)
\\ \ecart
\dis =\big(v(x/\ep)\,\zeta-(v\,b)(x/\ep)\big)\cdot\nabla_x\theta(x+t\,\zeta)=f(t,x,x/\ep),
\ea
\]
where
\[
f(t,x,y):=\big(v(y)\,\zeta-(v\,b)(y)\big)\cdot\nabla_x\theta(x+t\,\zeta)\quad\mbox{for }(t,x,y)\in[0,T]\times\R^d\times Y_d,
\]
is an admissible function in $C^0_c([0,T]\times\R^d;L^\infty_\sharp(Y_d))$ with compact support in $[0,T]\times\R^d\times Y_d$.
Moreover, we have $u_\ep(0,x)=\th(x)\,v(x/\ep)$ for $x\in\R^d$, where $\theta(x)\,v(y)\in C^0_c(\R^d;L^2_\sharp(Y_d))$ with compact support in $[0,T]\times\R^d$ is also an admissible function.
Hence, by the homogenization assumption the sequence $u_\ep(t,x)$ converges weakly in $L^2((0,T)\times\R^d)$ to $u(t,x)=\theta(x+t\,\zeta)\,\overline{v}$ solution to the homogenized equation \eqref{homtranequ}, {\em i.e.}
\[
\forall\,(t,x)\in [0,T]\times\R^d,\quad {\partial u\over\partial t}(t,x)-\zeta\cdot\nabla_x u(t,x)=\overline{f(t,x,\cdot)}=\big(\overline{v}\,\zeta-\overline{v\, b}\big)\cdot\nabla_x\theta(x+t\,\zeta).
\]
But directly from the expression $u(t,x)=\theta(x+t\,\zeta)\,\overline{v}$, we also deduce that
\[
\forall\,(t,x)\in [0,T]\times\R^d,\quad  {\partial u\over\partial t}(t,x)-\zeta\cdot\nabla_x u(t,x)=0.
\]
Equating the two former equations we get that for any $\theta\in C^1_c(\R^d)$,
\[
\forall\,(t,x)\in[0,T]\times\R^d,\quad \big(\overline{v}\,\zeta-\overline{v\, b}\big)\cdot\nabla_x\theta(x+t\,\zeta)=0,
\]
which implies that
\beq\label{vbze}
\overline{v\,b}=\overline{v}\,\zeta.
\eeq
\par
Now, let $\sigma$ be a non negative function in $L^1_\sharp(Y_d)$ with $\overline{\sigma}=1$, such that $\sigma(x)\,dx$ is an invariant measure for the flow $X$, or, equivalently, by Lemma~\ref{lem.invdenmea} applied with $\sigma_0=1$, the function $\sigma$ is invariant for $X$ with respect to Lebesgue's measure. Hence, for any $n\in\N$, the truncated function $\sigma\wedge n$ is also invariant for $X$. Equality \eqref{vbze} applied with $v=\sigma\wedge n\in L^\infty_\sharp(Y_d)$, yields
\[
\overline{(\sigma\wedge n)\,b}=\overline{\sigma\wedge n}\;\zeta\;\mathop{\longrightarrow}_{n\to\infty}\;\overline{\sigma\,b}=\overline{\sigma}\,\zeta=\zeta.
\]
Thus, we obtain the desired equality $\sfD_b=\{\zeta\}=\{\overline{b}\}$, which owing to Theorem~\ref{thm.Db} concludes the proof of Theorem~\ref{thm.SNChom}.
\end{proof}
\begin{arem}\label{rem.HY}
From equation \eqref{Ubdiv0} Hou and Xin~\cite{HoXi} used the ergodicity of the flow $X$ to deduce that $U(t,x,\cdot)$ is constant a.e. $(t,x)\in(0,T)\times\R^d$.
However, this condition is not necessary.
Indeed, the less restrictive condition used in the above proof is that $\sfD_b$ is reduced to the unit set~$\{\zeta\}$.
This combined with Lemma~\ref{lem.invdenmea} on invariant measures and functions leads us to equality \eqref{Uuzeta}, and allows us to conclude.
\end{arem}
\section{Comparison between the seven conditions}
\noindent
In the sequel we denote:
\par\medskip\noindent
\begin{tabular}{ll}
{\it Rec} & if there exist a $C^2$-diffeomorphism $\Psi$ on $Y_d$ and $\xi\in\R^d$ such that $\nabla\Psi\,b=\xi$ in $Y_d$.
\\*[.1cm]
{\it Erg} & if the ergodic condition~\eqref{Xerg} holds with an invariant probability measure for $X$,
\\
& which is absolutely continuous with respect to Lebesgue's measure,
\\*[.1cm]
{\em Asy-a.e.} & if there exist $\zeta\in\R^d$ such that $\; \dis\lim_{t\to\infty}{X(t,x)/t}=\zeta,\;$ a.e. $x\in Y_d$.
\\*[.1cm]
{\em Asy-e.} & if there exist $\zeta\in\R^d$ such that $\; \dis\lim_{t\to\infty}{X(t,x)/t}=\zeta,\;\;\forall\,x\in Y_d$.
\\*[.1cm]
$\#\sfC_b\!=\!1$ & if the unit set condition holds for Herman's set $\sfC_b$.
\\*[.1cm]
$\#\sfD_b\!=\!1$ & if the unit set condition holds for the set $\sfD_b$.
\\*[.1cm]
{\em Hom} & if the homogenized equation \eqref{homtranequ} holds when $b$ is divergence free in $\R^d$.
\end{tabular}
\begin{atheo}\label{thm.7con}
Let $b\in C^1_\sharp(Y_d)^d$ be a non null but possibly vanishing vector field such that there exists an invariant probability measure $\sigma_0(x)\,dx$ with $\sigma_0\in L^1_\sharp(Y_d)$, for the flow $X$ associated with $b$, or, equivalently, $\sfD_b\neq\mbox{{\rm\O}}$.
Then, we have a complete array (see Figure~\ref{array} below) of all the logical connections between the above seven conditions, in which:
\begin{itemize}
\item[-] A grey square means a tautology.
\\*[-.8cm]
\item[-] A square with $\Leftarrow$ means that the condition of the top line implies the condition of the left column, but not the converse in general.
\\*[-.8cm]
\item[-] A square with $\Uparrow$ means that the condition of the left column implies the condition of the top line, but not the converse in general.
\\*[-.8cm]
\item[-] A square with $\Leftrightarrow$ or $\Updownarrow$ means that the conditions of the top line and of the left column are equivalent.
\\*[-.8cm]
\item[-] A dark square means that the conditions of the top line and the left column cannot be compared in general.
\\*[-.8cm]
\item[-] Finally, if a square involves condition~{\it Hom}, then the other condition must be considered under the assumption that $b$ is divergence free in $\R^d$. 
\end{itemize}
\end{atheo}
\renewcommand{\arraystretch}{3.8}
\vskip -1.cm
\begin{figure}[t!]
\bigskip
{\footnotesize\centerline{
\begin{tabularx}{.82\textwidth}
{|>{\centering\arraybackslash}X|>{\centering\arraybackslash}X|>{\centering\arraybackslash}X|>{\centering\arraybackslash}X|>{\centering\arraybackslash}X|>{\centering\arraybackslash}X|>{\centering\arraybackslash}X|>{\centering\arraybackslash}X|}
\hline
& {\it Rec} & {\it Erg} & {\em Asy-a.e.} & {\em Asy-e.} & $\#\sfC_b\!=\!1$ & $\#\sfD_b\!=\!1$ & {\em Hom}
\\
\hline
{\it Rec} & \cellcolor{lightgray} & \cellcolor{darkgray} & $\Uparrow$ & $\Uparrow$ & $\Uparrow$ & $\Uparrow$ & $\Uparrow$
\\
\hline
{\it Erg} & \cellcolor{darkgray} & \cellcolor{lightgray} & $\Uparrow$ & \cellcolor{darkgray} &\cellcolor{darkgray} & $\Uparrow$ & $\Uparrow$
\\
\hline
{\em Asy-a.e.} & $\Leftarrow$ & $\Leftarrow$ & \cellcolor{lightgray} & $\Leftarrow$ & $\Leftarrow$ & $\Leftrightarrow$ & $\Leftrightarrow$
\\
\hline
{\em Asy-e.} & $\Leftarrow$ & \cellcolor{darkgray} & $\Uparrow$ & \cellcolor{lightgray} & $\Leftrightarrow$ & $\Uparrow$ & $\Uparrow$
\\
\hline
$\#\sfC_b\!=\!1$ & $\Leftarrow$ & \cellcolor{darkgray} & $\Uparrow$ & $\Updownarrow$ & \cellcolor{lightgray} & $\Uparrow$ & $\Uparrow$
\\
\hline
$\#\sfD_b\!=\!1$ & $\Leftarrow$ & $\Leftarrow$ & $\Updownarrow$ & $\Leftarrow$ & $\Leftarrow$ & \cellcolor{lightgray} & $\Leftrightarrow$
\\
\hline
{\em Hom} & $\Leftarrow$ & $\Leftarrow$ & $\Updownarrow$ & $\Leftarrow$ & $\Leftarrow$ & $\Updownarrow$ & \cellcolor{lightgray}
\\
\hline
\end{tabularx}
}}
\bigskip
\caption{Logical connections between the seven conditions}\label{array}
\end{figure}
\par\bigskip\noindent
\renewcommand{\arraystretch}{0}
\begin{arem}\label{rem.Dbe}
We may have both $\#\sfC_b=1$ and $\sfD_b=\mbox{{\rm\O}}$.
\par\noindent
To this end, consider a gradient field $b=\nabla u$ with $u\in C^2_\sharp(Y_d)$, such that $\nabla u\neq 0$ a.e. in~$Y_d$.
On the one hand, by virtue of \cite[Proposition~2.4]{BrHe} we have $\sfC_b=\{0\}$.
On the other hand, assume that there exists an invariant probability measure $\sigma(x)\,dx$ with $\sigma\in L^1_\sharp(Y_d)$, for the flow associated with $\nabla u$. Then, by virtue of Proposition~\ref{pro.divcurl} we have
\[
\int_{Y_d}\sigma(x)\,|\nabla u(x)|^2\,dx=\int_{Y_d}\sigma(x)\,\nabla u(x)\cdot\nabla u(x)\,dx=0,
\]
which implies that $\sigma=0$ a.e. in $Y_d$, a contradiction with $\overline{\sigma}=1$.
Therefore, we get that $\sfD_b=\mbox{{\rm\O}}$.
\end{arem}
\begin{proof}{ of Theorem~\ref{thm.7con}}
\par\smallskip\noindent
{\it Condition Rec.} By virtue of \cite[Corollary~4.1]{Bri2} condition~{\it Rec} implies condition~$\#\sfC_b\!=\!1$ which by Proposition~\ref{pro.Cb} is equivalent to condition~{\em Asy-e.}. Moreover, condition~{\em Asy-e.} clearly implies condition~{\em Asy-a.e.} which by Theorem~\ref{thm.Db} is equivalent to condition~$\#\sfD_b\!=\!1$, and by Theorem~\ref{thm.SNChom} is equivalent to condition~{\em Hom}. Therefore, condition {\it Rec} implies condition~{\em Asy-a.e.}, condition~{\em Asy-e.}, condition~$\#\sfC_b\!=\!1$, condition~$\#\sfD_b\!=\!1$, and condition~{\em Hom}.
\par
On the other hand, note that if the vector field $b$ vanishes, then condition~{\it Rec} cannot hold true. Otherwise, in equality $\nabla\Psi\,b=\zeta$ the constant vector $\zeta$ is necessarily nul, hence due to the invertibility of $\nabla\Psi$, $b$ is the nul vector field, which yields a contradiction.
Therefore, since all other conditions may be satisfied with a vanishing vector field $b$ according to the examples of \cite[Section~4]{BrHe} combined with Theorem~\ref{thm.Db} and Theorem~\ref{thm.SNChom}, condition~{\it Rec} cannot be deduced in general from any of the other six conditions.
\par\medskip\noindent
{\it Conditions {\it Rec} and {\it Erg} cannot be compared.} \cite[Corollary~4.1]{Bri2} provides a two-dimensional and a three-dimensional example in which condition~{\it Rec} holds true, but not condition~{\it Erg}.
\par\medskip\noindent
{\it Condition Erg.} By virtue of Birkhoff's theorem condition~{\it Erg} implies condition~{\em Asy-a.e.} which is equivalent to condition~$\#\sfD_b\!=\!1$ (by Theorem~\ref{thm.Db}) and is equivalent to condition~{\em Hom} (by Theorem~\ref{thm.SNChom}).
\par\medskip\noindent
{\it Conditions {\it Erg} and $\#\sfC_b\!=\!1$ cannot be compared.}
Since condition~{\it Rec} implies $\#\sfC_b\!=\!1$, but condition~{\it Rec} does not imply in general condition~{\it Erg} (by \cite[Corollary~4.1]{Bri2}), by a transitivity argument condition~$\#\sfC_b\!=\!1$ does not imply in general condition~{\it Erg}.
\par
On the other hand, extending the two-dimensional results of Oxtoby~\cite{Oxt} and Marchetto~\cite{Mar} to any dimension by a different approach, Example~\ref{exa.OMS} and Proposition~\ref{pro.OMS} below deal with a $d$-dimensional Stepanoff flow \cite[Section~4]{Ste} defined by
\beq\label{Sflow}
\left\{\ba{ll}
\dis {\partial S\over\partial t}(t,x)=b_S(S(t,x))=\rho_S(S(t,x))\,\xi, & t\in\R
\\ \ecart
S(0,x)=x\in\R^d, &
\end{array}\right.
\eeq
where $\rho_S$ is a non negative function in $C^1_\sharp(Y_d)$ with a finite positive number of roots in $Y_d$ and $\sigma_S:=1/\rho_S\in L^1_\sharp(Y_d)$, and where $\xi$ is a constant vector of $\R^d$ with incommensurable coordinates.
Under these conditions $\sigma_S(x)/\overline{\sigma_S}\;dx$ is the unique invariant probability measure on $Y_d$ for the flow $S$, which does not load the zero set of $\rho_S$, and $S$ is ergodic with respect to the measure $\sigma_S(x)/\overline{\sigma_S}\;dx$. Hence, condition {\it Erg} holds true with the probability measure $\sigma_S(x)/\overline{\sigma_S}\;dx$.
Moreover, Proposition~\ref{pro.OMS} shows that $\sfD_{b_S}=\{\zeta\}$ and $\sfC_{b_S}=[0,\zeta]$ with $\zeta=(1/\overline{\sigma_S})\,\xi\neq 0$.
Therefore, condition~{\it Erg} does not imply in general condition $\#\sfC_b\!=\!1$, or, equivalently, condition {\em Asy-e.}.
\par\medskip\noindent
{\it Conditions $\#\sfC_b\!=\!1$ and $\#\sfD_b\!=\!1$.} Since $\sfD_b$ is assumed to be non empty, condition~$\#\sfC_b\!=\!1$ clearly implies condition~$\#\sfD_b\!=\!1$.
\par
In contrast, as above mentioned the Stepanoff flow induces that $\sfD_{b_S}=\{\zeta\}$ and $\sfC_{b_S}=[0,\zeta]$ with $\zeta\in\R^d\setminus\{0\}$.
Alternatively, Example~\ref{exa.rhooDu} below provides a different class of two-dimensional vanishing vector fields $b$ such that $\sfD_b$ is a singleton, while $\sfC_b$ is a closed line set not reduced to a singleton. Therefore, condition~$\#\sfD_b\!=\!1$ does not imply in general $\#\sfC_b\!=\!1$.
\par\medskip\noindent
{\it Condition $\#\sfD_b\!=\!1$.} Since condition~$\#\sfC_b\!=\!1$ implies condition~$\#\sfD_b\!=\!1$, but $\#\sfC_b\!=\!1$ does not imply in general condition~{\it Erg}, by a transitivity argument condition~$\#\sfD_b\!=\!1$ does not imply in general condition~{\it Erg}. Moreover, since condition~$\#\sfC_b\!=\!1$ is equivalent to condition~{\em Asy-e.}, but condition~$\#\sfD_b\!=\!1$ does not imply in general $\#\sfC_b\!=\!1$, condition~$\#\sfD_b\!=\!1$ does not imply in general condition~{\em Asy-e.}.
\par\medskip\noindent
{\it Condition Hom.} Here, we assume that the vector field $b$ is divergence free in $\R^d$.
\par
On the one hand, consider the constant vector field $b=e_1$ in $\R^d$, which induces the flow
\[
X(t,x)=x+t\,e_1\quad\mbox{for }(t,x)\in\R\times\R^d.
\]
Then, any function $f\in L^1_\sharp(Y_d)$ independent of variable $x_1$ is invariant for the flow $X$ with respect to any invariant probability measure which is absolutely continuous with respect to Lebesgue's measure. Hence, the flow $X$ is not ergodic with respect to such an invariant probability measure. Moreover, we have immediately $\sfC_b=\sfD_b=\{e_1\}$. Therefore, condition {\em Hom} which is equivalent to condition $\#\sfD_b=1$ (by Theorem~\ref{thm.SNChom}), does not imply in general condition {\it Erg}.
\par
On the other hand, the two-dimensional divergence free Oxtoby example \cite[Section~2]{Oxt} combined with the uniqueness result of \cite[Theorem~1]{Oxt} (see Example~\ref{exa.OMS}) provides a flow which is ergodic with respect to Lebesgue's measure, and such that $\sfC_b$ is not a unit set (see Proposition~\ref{pro.OMS}). Therefore, since condition~{\it Erg} implies condition {\em Hom} (see, {\em e.g.}, \cite[Theorem~3.2]{HoXi}) condition {\em Hom} does not imply in general condition $\#\sfC_b=1$.
Finally, condition~{\em Hom} does not imply in general either condition~{\it Erg}, or condition~$\#\sfC_b\!=\!1$, or, equivalently, condition~{\em Asy-e.}.
\par\medskip\noindent
{\it The rest of the implications can be easily deduced from the former arguments.}
\end{proof}
\begin{aexam}\label{exa.OMS}
Oxtoby~\cite{Oxt} provided an example of a free divergence analytic two-dimensional vector field $b$ with $(0,0)$ as unique stationary point in $Y_2$, such that the associated flow $X$ is ergodic with respect to Lebesgue's measure, and such that Lebesgue's measure is the unique invariant measure for the flow $X$ among all the invariant probability measures which do not load the point $(0,0)$.
Oxtoby's example is actually based on a Stepanoff flow \eqref{Sflow}, where $\rho_S$ is a non negative function in $C^1_\sharp(Y_2)$ with $(0,0)$ as unique stationary point, and where $\xi$ is a constant vector of $\R^2$ with incommensurable coordinates. Stepanoff~\cite[Section~4]{Ste} proved that Birkhoff's theorem applies if $\sigma_S:=1/\rho_S$ is in $L^1_\sharp(Y_2)$, which is not incompatible with the analyticity for~$\rho_S$. A suitable candidate for $\rho_S$ is then the function (see \cite[Example~4.2]{BrHe} for another application)
\beq\label{rho0}
\rho_S(x):=\left(\sin^2(\pi x_1)+\sin^2(\pi x_2)\right)^{\beta_0}\quad\mbox{for }x\in Y_2,\quad\mbox{with }\beta_0\in(1/2,1).
\eeq
More generally, Oxtoby~\cite[Theorem~1]{Oxt} proved that any two-dimensional flow homeomorphic to a Stepanoff flow with a unique stationary point $x_0$, admits a unique invariant probability measure $\mu$ for the flow $S$ \eqref{Sflow} satisfying $\mu(\{x_0\})=0$, and that $S$ is ergodic with respect to~$\mu$. Twenty five years later, Marchetto~\cite[Proposition~1.2]{Mar} extended this result to any flow homeomorphic to a Stepanoff flow with a finite number of stationary points in $Y_2$.
\par
In what follows, we extend the two-dimensional results of \cite{Oxt,Mar} to any dimension $d\geq 2$, using a non ergodic and elementary approach based on some classical tools of PDE's analysis (mollification, truncation) combined with the characterization of invariant functions of Lemma~\ref{lem.invdenmea}.
\end{aexam}
\begin{apro}\label{pro.OMS}
Consider a $d$-dimensional, $d\geq 2$, Stepanoff flow $S$ \eqref{Sflow} where $\rho_S\in C^1_\sharp(Y_d)$ is non negative with a finite positive number of roots (the stationary points for $S$) in $Y_d$ and $\sigma_S:=1/\rho_S\in L^1_\sharp(Y_d)$, and where $\xi\in\R^d$ has incommensurable coordinates. Then, the measure $\sigma_S(x)/\overline{\sigma_S}\;dx$ is the unique invariant probability measure on $Y_d$ for the flow $S$, which does not load the zero set of $\rho_S$. The flow $S$ is also ergodic with respect to the measure $\sigma_S(x)/\overline{\sigma_S}\;dx$. Moreover, we have $\sfD_{b_S}=\{\zeta\}$ and $\sfC_{b_S}=[0,\zeta]$, where $\zeta:=1/\overline{\sigma_S}\;\xi$.
\end{apro}
\begin{arem}\label{rem.homeoS}
Similarly to \cite{Oxt,Mar} the result of Proposition~\ref{pro.OMS} actually extends to any flow which is homeomorphic to a Stepanoff flow.
\par\noindent
Indeed, let $\Psi$ be a $C^2$-diffeomorphism on $Y_d$ (see \cite[Remark~2.1]{BrHe}).
Define the flow $\hat{X}$ obtained through the homeomorphism $\Psi$ from the flow $X$ associated with a vector field $b\in C^1_\sharp(Y_d)^d$, by
\beq\label{hX}
\hat{X}(t,x):=\Psi\big(X(t,\Psi^{-1}(x))\big)\quad\mbox{for }(t,x)\in \R\times Y_d.
\eeq
According to \cite[Remark~2.1]{BrHe} the homeomorphic flow $\hat{X}$ is the flow associated with the vector field $\hat{b}\in C^1_\sharp(Y_d)^d$ defined by
\beq\label{hb}
\hat{b}(x)=\nabla\Psi(\Psi^{-1}(x))\,b(\Psi^{-1}(x))\quad\mbox{for }x\in Y_d.
\eeq
Now, let $\mu$ be a probability mesure on $Y_d$, and let $\hat{\mu}$ be the image measure of $\mu$ by $\Psi$ defined by
\[
\int_{Y_d}\varphi(x)\,d\hat{\mu}(x)=\int_{Y_d}\varphi(\Psi(y))\,d{\mu}(y)\quad\mbox{for }\varphi\in C^0_\sharp(Y_d).
\]
By \eqref{hX} we have
\beq\label{hXX}
\left\{\ba{l}
\dis \forall\,\varphi\in C^0_\sharp(Y_d),\quad \int_{Y_d}\varphi\big(\hat{X}(t,x)\big)\,d\hat{\mu}(x)=\int_{Y_d}\varphi\big(\Psi(X(t,y))\big)\,d{\mu}(y),
\\ \ecart
\dis \forall\,\rho\in C^0_\sharp(Y_d),\quad \hat{\mu}(\{\rho=0\})=\mu(\{\rho\circ\Psi=0\}),
\\ \ecart
\dis \forall\,f\in L^1_\sharp(Y_d),\ \hat{f}:=f\circ\Psi^{-1},\ \forall\,t\in\R,\quad \hat{f}\big(\hat{X}(t,x)\big)=f\big(X(t,\Psi^{-1}(x))\big)\;\;\mbox{a.e. }x\in Y_d.
\ea\right.
\eeq
Also note that, if $\mu$ is invariant for X, so is $\hat{\mu}$ for $\hat{X}$. 
Therefore, if the homeomorphic flow $\hat{X}$ is a Stepanoff flow $S$ satisfying the assumptions of Proposition~\ref{pro.OMS}, we easily deduce from~\eqref{hXX} that Proposition~\ref{pro.OMS} holds true for the flow $X$. Namely, there exists a unique invariant probability measure $\mu$ on $Y_d$ for the flow $X$, which does not load the zero set of $\rho_S\circ\Psi$. Moreover, the measure $\mu$ is absolutely continuous with respect to Lebesgue's measure with an a.e. positive density, and the flow $X$ is ergodic with respect to $\mu$.
\end{arem}
\begin{arem}\label{rem.ergS}
Assuming the uniqueness result of Proposition~\ref{pro.OMS}, the ergodicity of $\sigma_S(x)/\overline{\sigma_S}\;dx$ and equality $\sfC_{b_S} = [0,\zeta]$ can also be proved using standard arguments of ergodic theory. Indeed, let $\cE_{b_S}$ be the set of all the ergodic invariant probability measures for the flow $S$. Recall (see, {\em e.g.}, \cite[Theorem~2, Chapter~1]{CFS}) that $\cI_{b_S}={\rm conv}\,(\cE_{b_S})$, and that two elements in $\cE_{b_S}$ are either equal, or mutually singular. Now, if $\mu\in\cE_{b_S}$ satisfies $\mu(\{x_0\}) > 0$ for some zero of~$\rho_S$, then $\mu = \delta_{x_0}$ due to $\delta_{x_0} \in \cE_{b_S}$. Next, since $\sigma_S(x)/\overline{\sigma_S}\; dx$ is the unique invariant probability measure on $Y_d$ for the flow $S$, which does not load the zero set of $\rho_S$, it follows from equality $\cI_{b_S} ={\rm conv}\,(\cE_{b_S})$ that the flow $S$ is ergodic with respect to $\sigma_S(x)/\overline{\sigma_S}\; dx$.
Hence, $\cE_{b_S}$ is the finite set
\beq\label{EbS}
\cE_{b_S}=\big\{\sigma_S(x)/\overline{\sigma_S}\;dx\big\}\cup\big\{\delta_x:\rho_S(x)=0\big\}.
\eeq
Therefore, $\overline{\sigma_S\,b_S}/\overline{\sigma_S}=\zeta$ provides the unique non zero contribution in $\sfC_{b_S}$ through $\cE_{b_S}$, which by convex combination implies that $\sfC_{b_S}= [0,\zeta]$.
Equality $\cI_{b_S}={\rm conv}\,(\cE_{b_S})$ and property~\eqref{EbS} also give $\sfD_{b_S}=\{\zeta\}$.
\end{arem}
\begin{proof}{ of Proposition~\ref{pro.OMS}}
Assume that $\mu$ is an invariant probability measure for the flow $S$~\eqref{Sflow}, which does not load the zero set of $\rho_S$.
Then, by virtue of Proposition~\ref{pro.divcurl} the Borel measure~$\tilde{\mu}$ on $\R^d$ defined by \eqref{tmu} is solution to the equation
\[
{\rm div}(\tilde{\mu}\,b_S)={\rm div}(\rho_S\,\tilde{\mu}\,\xi)=0\quad\mbox{in }\cD'(\R^d).
\]
Hence, applying Lemma~\ref{lem.divtnuxi=0} below with the measure $\nu=\rho_S\,\mu$ which is connected to the measure $\tilde{\nu}=\rho_S\,\tilde{\mu}$ by~\eqref{tmu}, there exists a constant $c\in\R$ such that $\rho_S(x)\,d\mu(x)=c\,dx$ on $Y_d$, {\em i.e.}
\[
\forall\,\varphi\in C^0_\sharp(Y_d),\quad \int_{Y_d}\varphi(x)\,\rho_S(x)\,d\mu(x)=\int_{Y_d}c\,\varphi(x)\,dx.
\]
Then, we get that for any $n\geq 1$,
\beq\label{rhomun}
\forall\,\varphi\in C^0_\sharp(\R^d),\quad\int_{Y_d}{\varphi(x)\over\rho_S(x)+1/n}\,\rho_S(x)\,d\mu(x)=\int_{Y_d}c\,{\varphi(x)\over\rho_S(x)+1/n}\,dx.
\eeq
However, since measure $\mu$ does not load the finite zero set of $\rho_S$ in $Y_d$ (at this point this assumption is crucial), we have
\[
\left\{\ba{llll}
\dis {\varphi\,\rho_S\over\rho_S+1/n}\;\mathop{\longrightarrow}_{n\to\infty}\; \varphi & d\mu(x)\mbox{-a.e. in }Y_d,
& \mbox{with} & \dis \left|\,{\varphi\,\rho_S\over\rho_S+1/n}\,\right|\leq\|\varphi\|_\infty\in L^1_\sharp(Y_d,\mu)
\\ \ecart
 \dis {\varphi\over\rho_S+1/n}\;\mathop{\longrightarrow}_{n\to\infty}\; {\varphi\over\rho_S} & dx\mbox{-a.e. in }Y_d,
& \mbox{with} & \dis \left|\,{\varphi\over\rho_S+1/n}\,\right|\leq{\|\varphi\|_\infty\over \rho_S}\in L^1_\sharp(Y_d).
\ea\right.
\]
Therefore, passing to the limit as $n\to\infty$ owing to Lebesgue's theorem in \eqref{rhomun}, we get that
\[
\forall\,\varphi\in C^0_\sharp(Y_d),\quad \int_{Y_d}\varphi(x)\,d\mu(x)=\int_{Y_d}c\,\varphi(x)\,\sigma_S(x)\,dx.
\]
We thus obtain the equality $d\mu(x)=\sigma_S(x)/\overline{\sigma_S}\;dx$, which shows the uniqueness of an invariant probability measure for the flow $S$, which does not load the zero set of $\rho_S$. Conversely, $\sigma_ S\,b_S = \xi$ is clearly divergence free, which by Proposition~\ref{pro.divcurl} implies that $\mu$ is an invariant probability measure for the flow $S$. We have just proved that $\mu$ is the unique invariant probability measure for the flow $S$, which does not load the zero set of $\rho_S$.
\par
Now, let us prove that the flow $S$ is ergodic with respect to the measure $d\mu(x)=\sigma_S(x)/\overline{\sigma_S}\;dx$. To this end, let $f\in L^1_\sharp(Y_d)$ be an invariant function for the flow $S$ with respect to measure $\mu$.
Let $\sfT_n^\pm$ be the truncation functions at level $n\in\N$, defined by
\[
\sfT_n^\pm(t):=\big((\pm\,t)\vee 0)\big)\wedge n\quad\mbox{for }t\in\R.
\]
Then, the functions $T_n^\pm(f)\in L^\infty_\sharp(Y_d)$ are also invariant functions for the flow $S$ with respect to measure $\mu$.
We know that $\sigma_S(x)\,dx$ is an invariant measure for the flow $S$.
Hence, by virtue of Lemma~\ref{lem.invdenmea} the Radon measures $d\mu_n^\pm(x):=(T_n^\pm(f)\,\sigma_S)(x)\,dx$  are invariant for $S$, which by relation~\eqref{tmu} and Proposition~\ref{pro.divcurl} implies that
\[
\widetilde{\mu_n^\pm}\,b_S=\mu_n^\pm\,b_S=(T_n^\pm(f)\,\sigma_S\,b_S)(x)\,dx=T_n^\pm(f)(x)\,\xi\,dx
\]
are divergence free in~$\R^d$.
Therefore, applying Lemma~\ref{lem.divtnuxi=0} with measures $d\nu(y)=T_n^\pm(f)(y)\,dy$ which satisfy $\tilde{\nu}=\nu$, the functions $T_n^\pm(f)$ agree with constants $c_n^\pm\in\R$ a.e. in~$Y_d$.
However, since the sequences $T_n^\pm(f)$ converge strongly in $L^1_\sharp(Y_d)$ to the non negative and the non positive parts $f^\pm$ of $f$, the sequences~$c_n^\pm$ converge to some constants $c_\pm$ in $\R$. Hence, the function $f=f^+-f^-$ agrees with the constant $c_+-c_-$ a.e. in $Y_d$.
This proves the desired property.
\par
Next, since $\sigma_S(x)/\overline{\sigma_S}\;dx$ is the unique invariant probability measure on $Y_d$ for the flow~$S$, among the invariant probability measures which are absolutely continuous with respect to Lebesgue's measure, we have
\[
\sfD_{b_S}=\left\{\int_{Y_d}\rho_S(x)\,\xi\,\sigma_S(x)/\overline{\sigma_S}\;dx\right\}=\{1/\overline{\sigma_S}\;\xi\}.
\]
Note that the former equality can be alternatively deduced from the ergodicity of the flow $S$ combined with Theorem~\ref{thm.7con}.
\par\noindent
On the other hand, set $b_n:=b_S+1/n$ for $n\geq 1$. Since $\xi$ has incommensurable coordinates, we have (see \cite[Example~4.1]{BrHe})
\[
\sfC_{b_n}=\{\zeta_n\}\quad\mbox{where}\quad\zeta_n:=\left(\int_{Y_d}{dx\over \rho_S(x)+1/n}\,dx\right)^{-1}\zeta.
\]
Finally, since the function $\rho_S$ vanishes in $Y_d$, by virtue of \cite[Theorem~3.1]{BrHe} we obtain that
\[
\sfC_{b_S}=[0,\zeta],\quad\mbox{where}\quad \zeta=\lim_{n\to\infty}\zeta_n=1/\overline{\sigma_S}\;\xi.
\]
Note that the ergodic approach of Remark~\ref{rem.ergS} alternatively shows that $\sfC_{b_S}=[0,\zeta]$.
The proof of Proposition~\ref{pro.OMS} is now complete.
\end{proof}
\begin{alem}\label{lem.divtnuxi=0}
Let $\nu\in\cM_\sharp(Y_d)$, let $\tilde{\nu}\in\cM_{\rm loc}(\R^d)$ be the Borel measure on $\R^d$ connected to the measure $\nu$ by relation~\eqref{tmu}, and let $\xi\in\R^d$ be a vector with incommensurable coordinates. Assume that $\tilde{\nu}\,\xi$ is divergence free in~$\R^d$, {\em i.e.}
\beq\label{divtnuxi=0}
\forall\,\varphi\in C^\infty_c(\R^d),\quad \int_{\R^d}\xi\cdot\nabla\varphi(x)\,d\tilde{\nu}(x)=0.
\eeq
Then, there exists a constant $c\in\R$ such that $d\nu(y)=c\,dy$ on $Y_d$.
\end{alem}
\begin{arem}
In Lemma~\ref{lem.divtnuxi=0} the incommensurability of $\xi's$ coordinates is also a necessary condition to get \eqref{divtnuxi=0}.
Indeed, assume that there exists a non nul integer vector $k\in\Z^d\setminus\{0\}$ such that $k\cdot\xi=0$. Then, for any non constant $\Z$-periodic function $\theta\in C^1_\sharp(Y_1)$, the function $\big(\tau:x\mapsto\theta(k\cdot x)\big)$ belongs to $C^1_\sharp(Y_d)$,  $\tilde{\tau}(x)\,dx=\tau(x)\,dx$, and
\[
\forall\,x\in\R^d,\quad {\rm div}(\tau\,\xi)(x)=\theta'(k\cdot x)\,k\cdot\xi=0,
\]
so that the conclusion of Lemma~\ref{lem.divtnuxi=0} does not hold true.
\end{arem}
\begin{proof}{ of Lemma~\ref{lem.divtnuxi=0}} Let $(\phi_n)_{n\in\N}$ be a sequence of mollifiers in~$C^\infty_c(\R^d)$ with $\overline{\phi_n}~=1$.
Applying successively Fubini's theorem twice and \eqref{divtnuxi=0}, the convolution $\phi_n*\tilde{\nu}\in C^\infty(\R^d)$ satisfies for any $n\in\N$ and for any $\varphi\in C^\infty_c(\R^d)$,
\[
\ba{l}
\dis \int_{\R^d}\left(\int_{\R^d}\phi_n(x-y)\,d\tilde{\nu}(y)\right)\xi\cdot\nabla\varphi(x)\,dx=\int_{\R^d}\left(\int_{\R^d}\phi_n(x-y)\,\xi\cdot\nabla\varphi(x)\,dx\right)d\tilde{\nu}(y)
\\ \ecart
\dis =\int_{\R^d}\left(\int_{\R^d}\phi_n(x)\,\xi\cdot\nabla_x\varphi(x+y)\,dx\right)d\tilde{\nu}(y)=\int_{\R^d}\left(\int_{\R^d}\xi\cdot\nabla_y\varphi(x+y)\,d\tilde{\nu}(y)\right)\phi_n(x)\,dx
=0,
\ea
\]
or, equivalently,
\beq\label{divphintnu=0}
{\rm div}\big((\phi_n*\tilde{\nu})\,\xi\big)=\nabla(\phi_n*\tilde{\nu})\cdot\xi=0\quad\mbox{in }\R^d.
\eeq
Now, consider $\xi^1,\dots,\xi^{d-1}$ $(d\!-\!1)$ vectors in $\R^d$ such that $(\xi^1,\dots,\xi^{d-1},\xi)$ is an orthogonal basis of $\R^d$, and let $\Lambda$ be the matrix in $\R^{(d-1)\times d}$ whose lines are the vectors $\xi^1,\dots,\xi^{d-1}$, {\em i.e.} its entries are given by $\Lambda_{ij}=\xi^i_j$ for $(i,j)\in\{1,\dots,d\!-\!1\}\times\{1,\dots,d\}$.
Then, make the linear change of variables
\[
\ba{rll}
\R^d & \to & \R^d
\\ \ecart
x & \mapsto & y=(\Lambda x,\xi\cdot x)=(\xi^1\cdot x,\dots,\xi^{d-1}\cdot x,\xi\cdot x).
\ea
\]
Since \eqref{divphintnu=0} means that $(\phi_n*\tilde{\nu})(x)$ is independent of the variable $y_d=\xi\cdot x$, it follows that there exists a function $\theta_n\in C^\infty(\R^{d-1})$ such that
\[
\forall\,x\in\R^d,\quad (\phi_n*\tilde{\nu})(x)=\theta_n(\Lambda  x).
\]
Moreover, due to \eqref{tmu} and the $\Z^d$-periodicity of $(\phi_n)_\sharp$, we have for any $x\in\R^d$ and $k\in\Z^d$,
\[
\ba{l}
\dis (\phi_n*\tilde{\nu})(x+k)=\int_{\R^d}\phi_n(x+k-y)\,d\tilde{\nu}(y)=\int_{Y_d}(\phi_n)_\sharp(x+k-y)\,d\nu(y)
\\ \ecart
\dis =\int_{Y_d}(\phi_n)_\sharp(x-y)\,d\nu(y)=\int_{\R^d}\phi_n(x-y)\,d\tilde{\nu}(y)=(\phi_n*\tilde{\nu})(x),
\ea
\]
which implies that the function $\phi_n*\tilde{\nu}$ is also $\Z^d$-periodic.
As a consequence, the regular function $\theta_n$ satisfies the periodicity condition
\[
\forall\,k\in\Z^d,\ \forall\,x\in\R^{d-1},\quad \theta_n(x+\Lambda k)=\theta_n(x).
\]
Hence, by virtue of the density Lemma~\ref{lem.densubgRd} below we get that $\theta_n$ is a constant $c_n\in\R$, and thus $\phi_n*\tilde{\nu}=c_n$ in~$\R^d$.
Therefore, by Fubini's theorem we have for any $\varphi\in C^\infty_c(\R^d)$,
\[
\dis \int_{\R^d}c_n\,\varphi(x)\,dx=\int_{\R^d}\left(\int_{\R^d}\phi_n(x-y)\,d\tilde{\nu}(y)\right)\varphi(x)\,dx
=\int_{\R^d}\left(\int_{\R^d}\phi_n(x-y)\,\varphi(x)\,dx\right)d\tilde{\nu}(y),
\]
where the function $\big(y\mapsto \int_{\R^d}\phi_n(x-y)\,\varphi(x)\,dx\big)$ converges uniformly to $\varphi$ on $\R^d$ as $n\to\infty$, whose support is included in a fixed compact set of~$\R^d$, and which is bounded uniformly by $\|\varphi\|_\infty$.
Therefore, passing to the limit as $n\to\infty$ thanks to Lebesgue's theorem with respect to measure~$\tilde{\nu}$, we get that the sequence $(c_n)_{n\in\N}$ converges to some $c\in\R$, and that
\[
\forall\,\varphi\in C^\infty_c(\R^d),\quad \int_{\R^d}c\,\varphi(x)\,dx=\int_{\R^d}\varphi(y)\,d\tilde{\nu}(y).
\]
Hence, we deduce the equality $d\tilde{\nu}(x)=c\,dx$ on $\R^d$, or, equivalently, $d\nu(y)=c\,dy$ on $Y_d$ by virtue of Remark~\ref{rem.tmu}.
This concludes the proof of Lemma~\ref{lem.divtnuxi=0}.
\end{proof}
\begin{alem}\label{lem.densubgRd}
Let $\xi$ be a vector in $\R^d$ for $d\geq 2$, with incommensurable coordinates, let $\xi^1,\dots,\xi^{d-1}$ be $(d\!-\!1)$  vectors in $\R^d$ such that
$(\xi^1,\dots,\xi^{d-1},\xi)$ is an orthogonal basis of $\R^d$, and let $\Lambda$ be the matrix in $\R^{(d-1)\times d}$ whose lines are the vectors $\xi^1,\dots,\xi^{d-1}$.
Then, the lattice $\Lambda\,\Z^d$ is dense in~$\R^{d-1}$.
\end{alem}
\begin{proof}{} Lemma~\ref{lem.densubgRd} follows easily from \cite[Proposition~6 \& Corollary, Section~VII.7]{Bou} which leads one to Kronecker's approximation theorem \cite[Proposition~7, Section~VII.7]{Bou}. For the reader's convenience we propose a more direct proof.
\par\noindent
Since matrix $\Lambda$ has rank $(d\!-\!1)$  and $\ker(\Lambda)=\R\,\xi$, we may assume, up to reorder the vectors, that the vectors $\Lambda e_1,\dots,\Lambda e_{d-1}$ are linearly independent and that there exist $d$ real numbers $\alpha_1,\dots,\alpha_{d-1},\alpha$ satisfying
\beq\label{Laei}
\Lambda e_d=\sum_{i=1}^{d-1}\alpha_i\,\Lambda e_i\quad\mbox{and}\quad e_d-\sum_{i=1}^{d-1}\alpha_i\,e_i=\alpha\,\xi.
\eeq
Replacing the vector $e_d$ in the first equality of \eqref{Laei} and using that $\Lambda\xi=0$, we get that
\[
\Lambda\,\Z^d=\sum_{i=1}^d\, \Z\,\Lambda e_i=\sum_{i=1}^{d-1}\,(\Z+\alpha_i\,\Z)\,\Lambda e_i.
\]
Assume that there exists $j\in\{1,\dots,d-1\}$ such that the set $(\Z+\alpha_j\,\Z)$ is not dense in $\R$, or, equivalently, $\alpha_j\in\Q$.
Taking the $j$-th and $d$-th coordinates in the second equality of \eqref{Laei}, it follows that $\xi_j+\alpha_j\,\xi_d=0$, which contradicts the incommensurability of $\xi$'s coordinates. Therefore, the set $\Lambda\,\Z^d$ is dense in $\R^{d-1}$, which concludes the proof.
\end{proof}
\begin{aexam}\label{exa.rhooDu}
Consider a two-dimensional vector field $b=\rho_0\,R_\perp\nabla u$ such that $\rho_0\in C^1_\sharp(Y_2)$ is a.e. positive in $Y_2$ and does vanish in $Y_2$, and such that $\nabla u\in C^1_\sharp(Y_2)^2$ does not vanish in $Y_2$ and $\overline{\nabla u}$ has incommensurable coordinates. Also assume that $\sigma_0:=1/\rho_0\in L^1_\sharp(Y_2)$. An example of such a function is given by \eqref{rho0}. Note that, by virtue of Proposition~\ref{pro.divcurl} the probability measure $\sigma_0(x)/\overline{\sigma_0}\;dx$ is invariant for the flow $X$ associated with $b$.
\par
Now, let $\sigma\in L^1_\sharp(Y_2)$ be a non negative function with $\overline{\sigma}=1$, such that ${\rm div}(\sigma b)=0$ in $\R^2$.
By Proposition~\ref{pro.divcurl} $\sigma(x)\,dx$ is an invariant probability measure for the flow $X$.
Hence, by Fubini's theorem we have for any $T>0$,
\beq\label{bTX}
\int_{Y_2}\sigma(x)\,b(x)\,dx = {1\over T} \int_0^T\left(\int_{Y_2}\sigma(x)\,b(X(t,x))\,dx\right)dt= \int_{Y_2}\left({X(T,x)-x\over T}\right)\sigma(x)\,dx.
\eeq
On the other hand, since the function $\rho_0$ does vanish in $Y_2$ together with $\rho_0>0$ a.e. in $Y_2$, from \cite[Lemma~3.1]{BrHe} applied with the invariant probability measure $d\mu(x):=\sigma_0(x)/\overline{\sigma_0}\;dx$, we deduce that
\[
\lim_{T\to\infty}{X(T,x)\over T} = \zeta := {\overline{\sigma_0\,b}\over \overline{\sigma_0}}={R_\perp\overline{\nabla u}\over\overline{\sigma_0}}\neq (0,0)
\quad\mbox{a.e. }x\in Y_2.
\]
Therefore, passing to the limit $T\to\infty$ in equality \eqref{bTX} thanks to Lebesgue's theorem, we get that for any invariant probability measure $\sigma(x)\,dx$ with $\sigma\in L^1_\sharp(Y_2)$,
\[
\int_{Y_2}\sigma(x)\,b(x)\,dx = \zeta\neq (0,0),
\]
which thus implies that $\sfD_b = \{\zeta\}$. However, by virtue of \cite[Corollary~3.4] {BrHe} we obtain that $\sfC_b = [0,\zeta]$.
Therefore, we have $\#C_b=\infty$, while $\#\sfD_b=1$.
\end{aexam}
\begin{arem}\label{rem.K}
The result of Example~\ref{exa.rhooDu} can be deduced from the Proposition~\ref{pro.OMS} combined with Remark~\ref{rem.homeoS}, using Kolmogorov's theorem \cite{Kol} (see, {\em e.g.}, \cite[Lecture~11]{Sin}, and see also \cite[Theorem~2.1]{Tas} for an elementary proof when one of the coordinates of the vector field does not vanish).
Indeed, since the divergence free field $R_\perp \nabla u$ of Example~\ref{exa.rhooDu} does not vanish in $Y_2$, by virtue of Kolmogorov's theorem there exists a $C^1$-diffeomorphism on $Y_2$ which transforms the flow $X$ associated with the vector field $b=\rho_0\,R_\perp\nabla u$, to a Stepanoff flow satisfying the assumptions of Proposition~\ref{pro.OMS} provided the zero set of $\rho_0$ is finite. Therefore, Remark~\ref{rem.homeoS} allows us to conclude.
\end{arem}
We can extend Example~\ref{exa.rhooDu} to the following variant of \cite[Corollary~3.4]{BrHe}, which provides a general framework where the sets $\sfC_b$ and $\sfD_b$ may differ.
\begin{apro}\label{pro.CbneDb}
Let $b=\rho\,\Phi\in C^1_\sharp(Y_2)^2$ be a vector field, where $\rho\in C^1_\sharp(Y_2)$ is a non negative function with a positive finite number of roots, and where $\Phi\in C^1_\sharp(Y_2)^2$ is a non vanishing vector field. Also assume that there exists a function $u\in C^1(Y_2)$ with $\nabla u\in C^0_\sharp(Y_2)^2$, such that $\overline{\nabla u}$ has incommensurable coordinates and $\Phi\cdot\nabla u=0$ in $Y_2$.
Then, the exists a vector $\zeta\in\R^2$ such that $\sfC_b=[0,\zeta]$, together with $\sfD_b={\rm\O}$ or $\sfD_b=\{\zeta\}$.
\end{apro}
\begin{proof}{}
First of all, define for $n\geq 1$, the function $\rho_n:=\rho+1/n>0$, and the vector field $b_n:=\rho_n\,\Phi$. By the equality $\Phi\cdot\nabla u=0$ in $Y_2$, we get that $u$ is an invariant function for the flow $X_n$ associated with the vector field $b_n$, with respect to Lebesgue's measure. Then, following the proof of \cite[Corollary~3.4]{BrHe}, from the ergodic case of \cite[Theorem~3.1]{Pei} and the incommensurability of $\overline{\nabla u}$'s coordinates, we deduce that there exists a vector $\zeta_n\in\R^2$ such that $\sfC_{b_n}=\{\zeta_n\}$.
\par
On the one hand, since the function $\rho$ vanishes in~$Y_2$, by the second case of \cite[Theorem~3.1]{BrHe} it turns out that the sequence $(\zeta_n)_{n\geq 1}$ converges to some $\zeta\in\R^2$, and that $\sfC_b=[0,\zeta]$.
\par
On the other hand, assume that the set $\sfD_b$ is non empty. Then, there exists an invariant probability measure $\sigma(x)\,dx$ with $\sigma\in L^1_\sharp(Y_2)$, for the flow $X$ associated with the vector field $b$, {\em i.e.} $\sigma(x)/\overline{\sigma}\,dx\in\cI_b$.
Following the proof of \cite[Corollary~3.3]{BrHe} define the probability measure $\mu_n$ by
\[
d\mu_n(x):=C_n\,{\rho(x)\over\rho_n(x)}\,\sigma(x)\,dx\quad\mbox{where}\quad C_n:=\left(\int_{Y_d}{\rho(x)\over\rho_n(y)}\,\sigma(y)\,dy\right)^{-1}.
\]
Note that $C_n<\infty$, since $\rho\,\sigma$ is non negative and not nul a.e. in $Y_2$.
Due to $\sigma(x)/\overline{\sigma}\,dx\in\cI_b$, by Proposition~\ref{pro.divcurl} we have
\[
\forall\,\varphi\in C^1_\sharp(Y_2),\quad \int_{Y_d}b_n(x)\cdot\nabla\varphi(x)\,d\mu_n(x)=C_n\int_{Y_d}b(x)\cdot\nabla\varphi(x)\,\sigma(x)\,dx=0,
\]
which again by Proposition~\ref{pro.divcurl} implies that $\mu_n\in\cI_{b_n}$.
This combined with $\sfC_{b_n}=\{\zeta_n\}$ yields
\[
\zeta_n=\int_{Y_d}b_n(x)\,d\mu_n(x)=C_n\int_{Y_d}\,b(x)\,\sigma(x)\,dx=C_n\,\overline{\sigma\,b}
\]
which is actually independent of $\sigma$.
Due $\rho>0$ a.e. in $Y_2$, by Lebesgue's theorem we get that the sequence $(C_n)_{n\geq 1}$ converges to $\overline{\sigma}=1$.
Hence, we deduce that
\[
\zeta=\lim_{n\to\infty}\zeta_n=\overline{\sigma\,b}
\]
which is also independent of $\sigma$. Therefore, we obtain that $\sfD_b=\{\zeta\}$, which concludes the proof of Proposition~\ref{pro.CbneDb}.
\end{proof}

\end{document}